\documentclass[11pt]{amsart}
\usepackage{etex}
\usepackage{amsmath,amssymb,amsthm}
\usepackage{amsfonts,amscd,cancel,stackrel}
\usepackage{array}
\usepackage[all]{xy}
\usepackage{euscript}
\usepackage{memhfixc}
\usepackage{latexsym}
\usepackage{tikz, color}
\usepackage{tikz-3dplot}
\usetikzlibrary{shapes,backgrounds}

\usepackage[colorlinks,citecolor=red,pagebackref,hypertexnames=false]{hyperref}
\usepackage{graphics,enumerate}

\usepackage{verbatim}
\usepackage{algorithm2e}
\usepackage{blkarray,multicol, float}
\usepackage{multirow,mathtools, leftidx}

\usepackage{enumerate}
\usepackage{esint}
\usepackage{caption}
\usepackage{subcaption} 
\usepackage{relsize}

\RequirePackage{fix-cm}
\DeclareMathSizes{10}{6}{5}{4}

\DeclareFontFamily{OT1}{rsfs}{}
\DeclareFontShape{OT1}{rsfs}{n}{it}{<-> rsfs10}{}
\DeclareMathAlphabet{\mathscr}{OT1}{rsfs}{n}{it}

\newcommand{\FF}{\mathbb F}

\newcommand\lM{\begin{bmatrix}}
\newcommand\rM{\end{bmatrix}}

\newcommand{\V}{\mathcal{V}}

\theoremstyle{Theorem}
\newtheorem{theorem}{Theorem}[section]
\newtheorem{corollary}[theorem]{Corollary}
\newtheorem{proposition}[theorem]{Proposition}
\newtheorem{lemma}[theorem]{Lemma}
\theoremstyle{definition}
\newtheorem{definition}[theorem]{Definition}

\newtheorem{example}[theorem]{Example}

\newcommand{\sC}{\mathscr{C}}
\newcommand{\sP}{\mathscr{P}}
\newcommand{\sB}{\mathscr{B}}

\tdplotsetmaincoords{45}{45}

\title{Polarization of Neural Rings} 

     \author{Sema G\"unt\"urk\"un}
     \address{University of Michigan, Department of Mathematics, USA}
     \email{gunturku@umich.edu}

     \author{Jack Jeffries*}
     \address{University of Michigan, Department of Mathematics, USA}
     \email{jackjeff@umich.edu}
     \thanks{* The NSF Postdoctoral Research Fellowship
DMS~\#1606353}
     
      \author{Jeffrey Sun**}
     \address{University of Michigan, Department of Mathematics, USA}
     \email{jeffjeff@umich.edu}
     \thanks{** The NSF grant DMS~\#1306992 of Professor Harm Derksen}


\begin{document}
\maketitle 

 \begin{abstract}
The ``neural code'' is the way the brain characterizes, stores, and processes information. Unraveling the neural code is a key goal of mathematical neuroscience. Topology, coding theory, and, recently, commutative algebra are some the mathematical areas that are involved in analyzing these codes. Neural rings and ideals are algebraic objects that create a bridge between mathematical neuroscience and commutative algebra. A neural ideal is an ideal in a polynomial ring that encodes the combinatorial firing data of a neural code. Using some algebraic techniques one hopes to understand more about the structure of a neural code via neural rings and ideals. In this paper, we introduce an operation, called ``polarization,'' that allows us to relate neural ideals with squarefree monomial ideals,  which are very well studied and known for their nice behavior in commutative algebra. 
\end{abstract}



\section{Introduction}

 The well-known Stanley-Reisner correspondence relates simplicial complexes to squarefree monomial ideals.
 Many combinatorial properties of simplicial complexes are encoded in algebraic data related to the associated ideal. For example, Hochster's formula gives a relation between Betti numbers of the ideal and the dimensions of the homologies of the simplicial complex.
Recently, an extension of this correspondence has been studied. 
 Instead of simplicial complexes, which are particular collections of subsets of a set, one studies \emph{combinatorial codes}, also known as \emph{neural codes}, which are arbitrary collections of subsets of a set.
 One associates to each neural code a \emph{pseudomonomial ideal}, an ideal generated by products of distinct elements of the form $x_i$ and  $1-x_j$, called the \emph{neural ideal} of the code.
 Roughly, the neural ideal of a neural code is the ideal of a variety over $\FF_2$ whose points are in bijection with elements of the code, with the boolean relations removed.
 Combinatorial data from the code relate to algebraic data of the ideal. In particular, the primary decomposition and a distinguished generating set each relate to important combinatorial data about the code. However, neural ideals lack many of the desirable algebraic properties of Stanley-Reisner ideals. In particular, they are not graded nor realizable in a local setting.
	
	In this paper, we provide a technique to relate neural ideals to squarefree monomial ideals, a process that we call \emph{polarization}. We show that our notion of polarization behaves in a similar way to polarization of monomial ideals; in particular, ``depolarization'' is given by quotienting out by a regular sequence. We apply our notion of polarization to show that the na\"ive analogue of the Taylor resolution gives a free resolution of the neural ideal. We also apply this to define a meaningful substitute for the notion of minimal resolutions in this context. Minimal primes persist through polarization, but extra minimal primes may appear in the primary decomposition of the polarized ideal. We provide a description of the minimal primary decomposition of the polarized neural ideal in terms of basic data about the neural code.
	
	We conclude this introduction by describing the original motivation for neural ideals. The experimental work of O'Keefe \cite{O'Keefe} on freely moving rats yielded deep insights into how the brain encodes stimuli via neural activity. His experiments showed that individual neurons, the \emph{place cells}, would fire when the rat entered a specific area, the \emph{place field} of the neuron. Different place cells correspond to different place fields, and these place fields may overlap or cover other place fields. Thus, at any given point, some place cells will fire and some will not. Moreover, place fields are generally convex regions. Other experiments have shown that neurons detecting various other stimuli roughly fit into a similar framework \cite{orientation, cat1, cat2}.
		
		An important question is to determine how much information about the ambient space is encoded in neural firing data. A key first step in this direction is \cite{CI}, where it is shown that, under mild assumptions, the topology of the ambient space can be recovered from the simplicial complex
		consisting of sets $S$ of place cells such that there is a point where each neuron in $S$ fires. Later, in \cite{CIVY} the authors consider a finer combinatorial object---the collection of sets $S$ of place cells such that there is a point where each neuron and $S$ \emph{and no neuron not in $S$} fires. Nerual ideals and their connection to combinatorial codes are first studied there as a tool to understand the relationships between place fields.

\section{Background}\label{sec-2}
In this section we briefly introduce the study of neural rings and ideals, some basic properties of them. Then we describe the polarization operation of monomial ideals.
\subsection{Neural Rings}


Set $[n] = \{1,\dots,n\}$. A \textbf{combinatorial code} $\sC$ on $[n]$ is an arbitrary subset of  $\{0,1\}^{[n]} = \FF_2^{[n]}$. An element of a combinatorial code is a \textbf{codeword}. We identify a code with a subset of the powerset $2^{[n]}$ by identifying a codeword $\omega=(\omega_1,\dots,\omega_n) \in \{0,1\}^{[n]}$ with the subset $\{ i \in [n] \ | \ \omega_i =1 \}\subseteq [n]$. We define a combinatorial code on a subset on $[n]$ in the same way. In keeping with the neuroscience motivation and previous literature, we refer to elements of $[n]$ as \textbf{neurons}, and we call a combinatorial code a \textbf{neural code}.

Given a finite collection of subsets $\mathcal{U}=\{U_1,\dots,U_n\}$ of an ambient space $X$, the \textbf{code associated to the cover $\mathcal{U}$} is the collection of subsets $S\subseteq [n]$ such that there is some $p\in X$ with $p\in U_i$ for $i\in S$ and $p\notin U_j$ for $j \notin S$. In the Nobel Prize-winning work of O'Keefe \cite{O'Keefe}, experiments showed that in animals, individual hippocampus neurons would fire when the animal would pass through a fixed convex region in space, called the \textbf{place field} of that neuron. Then, the set of possible firing patterns---sets of neurons that can fire simultaneously---is the neural code associated to the collection of place fields in the space. 

There has been much recent work on determining which neural codes can arise in this way, and figuring out how to reconstruct data about the ambient space from this information; see \cite{CI,CIVY,Curto-Youngs,MRC,REU1,REU2}. To this end, the following tool was introduced.

\begin{definition}\cite{CIVY} The \textbf{neural ring} $R_\sC$ associated to the neural code $\sC$, is defined to be the quotient ring \[R_\sC = \FF_2[x_1,\dots,x_n] / I_\sC\] where
\[I_\sC = \langle f \in \FF_2[x_1,\dots,x_n]  \ \vert \ f(c) = 0 \ \text{for all} \ c\in \sC \rangle\] is an ideal in $\FF_2[x_1,\dots,x_n]$.
\end{definition}

Notice that the neural code $\sC$ can be realized as the variety $\V(I_\sC)\subseteq \FF_2^{[n]}=\{0,1\}^{[n]}$, and $I_{\sC}$ as the vanishing ideal of $\sC$.  For any neural code $\sC$, the ideal $I_\sC$ contains the ideal $\sB = \langle x_1^2-x_1,\dots, x_n^2-x_n\rangle $ generated by the Boolean relations, which are verified by any point in $\FF_2^{[n]}$. Then the ideal $I_\sC$ can be decomposed as
$I_\sC  = \sB + J_\sC$ where 
\[J_\sC = \Big\langle \prod\limits_{\{i \vert c_i = 1 \}} x_i \cdot\prod\limits_{\{i \vert c_i = 0 \}}(1-x_j) \ \Big| \ c \not\in \sC \Big\rangle \]
The ideal $J_\sC$ is said to be the \textbf{neural ideal} associated to the neural code~$\sC$.

\begin{example} \label{ex-rats} Consider the following configuration of three place fields in the plane, taken from \cite{Youngs2}:

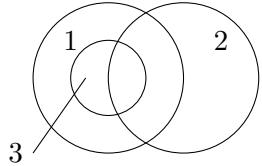
\begin{center}
\begin{tikzpicture}
    \draw {(0,0) circle (1cm)} node at (-0.5,0.5) {$1$};
   \draw {(1,0) circle (1cm)} node at (1.5,0.5) {$2$};
    \draw {(0,0) circle (0.5cm)}  (-0.3, 0) -- (-1, -1) node [left] {$3$};
    \end{tikzpicture}
\end{center}

Any subset of the regions intersect, but because region 3 is contained in region 1, there is no point where neuron 3 can fire without neuron 1. The neural code of the above configuration is 
\[\sC = \{ \emptyset, \{1\}, \{2\},\{1,2\},\{1,3\},\{1,2,3\}\} = \{000, 100, 010, 110, 101, 111\},\] so  $001, 011 \not\in \sC$. Thus the corresponding neural ideal is
\[J_\sC =\langle (1-x_1)(1-x_2)x_3, (1-x_1)x_2x_3\rangle =\langle (1-x_1)x_3\rangle .\]
\end{example}

A \textbf{pseudomonomial} in a polynomial ring is a product
\[
\prod_{i\in\sigma}x_i\prod_{j\in\tau}(1-x_j) \in R
\]
where $\sigma, \tau \subseteq [n]$ and $\sigma\cap\tau = \emptyset$.  A \textbf{pseudomonomial ideal} is an ideal generated by a finite set of pseudomonomials; evidently, a neural ideal is a pseudonomial ideal. We say a pseudomonomial $f$ is \textbf{minimal} in $I$ if no proper divisor of $f$ is contained in $I$. Furthermore, if $\{f_1,\ldots,f_\ell\}$ is the set of \textit{all} minimal pseudomonomials in $I$, and we write $I$ in the following form,
\[
I = \langle f_1,\ldots, f_\ell\rangle
\]
we say that $I$ is in \textbf{canonical form}.

The canonical form  of an ideal encodes all of the intersection and covering information of a collection of sets. To wit:
\begin{theorem}\cite{CIVY} Let $\sC$ be the neural code associated to a collection of subsets $\mathcal{U}=\{U_1,\dots,U_n\}$ of a space $X$. Let $J_{\sC}=\langle f_1,\dots, f_t \rangle$ be the neural ideal of $\sC$ in canonical form. Then $\prod_{i \in \sigma} x_i \prod_{j\in \tau} (1-x_j)$ occurs as some $f_k$ if and only if
\[ \bigcap_{i\in \sigma} U_i \subseteq \bigcup_{j\in \tau} U_j \]
and if this containment fails if $\sigma$ or $\tau$ is replaced by a proper subset. Here, we interpret the union over the empty set to be the empty set, and the intersection over the empty set to be the ambient space $X$.
\end{theorem}

To return to the example \ref{ex-rats}, the ideal $J_\sC$ has canonical form
\[
CF(J_\sC) = \langle (1-x_1)x_3 \rangle.
\]
This relation is showing that region 1 contains region 3.

For any $\alpha\in\{0,1,\ast\}^{[n]}$, we define \textbf{the interval associated to $\alpha$} as
\[ V_\alpha = \{ c \in  \{0,1\}^{[n]} \ | \ c_i=\alpha_i \ \text{for} \ \alpha_i\neq \ast  \}, \]
and \textbf{the pseudomonomial prime associated to $\alpha$} as
\[
\mathbf{p}_\alpha =\langle \{x_i \ | \ \alpha_i = 0\},\{1 - x_i \ | \ \alpha_i = 1\}\rangle .
\]

The interval associated to $\alpha$ is the Boolean interval in $\{0,1\}^{[n]}$ consisting of codewords that agree with $\alpha$ in the non-$\ast$ coordinates. One may think of $\ast$ as a ``wildcard'' coordinate in this construction. The pseudomonomial prime associated to $\alpha$ is the prime of $R$ that defines the variety of $V_{\alpha}$ in $\FF_2^n$ after adding the Boolean relations.

\begin{theorem}[CIVY \cite{CIVY}]\label{thm-PM-PD}
\[
J_\sC = \bigcap_{\{\alpha \, | \, V_\alpha \subseteq \sC\}} \mathbf{p}_\alpha.
\]
is the unique irredundant primary decomposition of $J_\sC$. In particular, a pseudomonomial ideal has a pseudomonomial primary decomposition.
\end{theorem}

In above example \ref{ex-rats}, the primary decomposition of $J_\sC$ is
\[
J_\sC = \langle x_3 \rangle \cap \langle 1-x_1\rangle .
\]

\subsection{Polarization Of Monomial Ideals}

In study of monomial ideals, an operation called ``\textit{polarization}" is used to reduce the monomial ideals to the squarefree monomial ideals. The advantage of reducing the squarefree monomials is that one can use the Stanley-Reisner correspondence which provides a very nice correlation between squarefree monomial ideals and simplicial complexes.

Let $m = x_1^{a_1}\cdots x_n^{a_n}$ be a monomial in $R = K[x_1,\dots,x_n]$ where $a_i \geq 0$ for all $i=1,\dots,n$.
The \textbf{polarization of a monomial $m$} is defined as 
\[\mathcal{P}(m) = \prod\limits_{i=1}^{n} x_i y_{i2}\cdots y_{i a_i}\]
which is a squarefree monomial in an extended polynomial ring $$K[x_1, y_{12},\dots,y_{1 a_1}, x_2, y_{22},\dots,y_{2a_2},\dots,x_n, y_{n2},\dots,y_{n a_n}].$$
Note that if $a_i\leq 1$ for any $1\leq i \leq n$, so that $m$ is already squarefree then $\mathcal{P}(m) = m$.  

The \textbf{ polarization of a monomial ideal} $I = (m_1,\dots,m_r)$ is 
\[\mathcal{P}(I) = \langle \mathcal{P}(m_1),\dots, \mathcal{P}(m_r) \rangle ,\] 
and clearly $\mathcal{P}(I)$ lives in a larger ring $S = R[y_{ij}]$ where the number of variables $y_{ij}$, for $i=1,\dots,n$, depends on the maximum degrees of $x_i$ in all minimal generators $m_l$.
\begin{example} Let $I = \langle x_1^2x_2^2, x_2^4x_3\rangle $ be a monomial ideal in $R = K[x_1,x_2,x_3]$. Then its polarization is the ideal
$\mathcal{P}(I ) =  \langle x_1y_{12}x_2y_{22},  x_2y_{22}y_{23}y_{24}x_3 \rangle$
 inside the larger polynomial ring 
 \[S = R[y_{12},  y_{22},  y_{23}, y_{24}] = K[x_1,  y_{12},  x_2,  y_{22},  y_{23},  y_{24}, x_3].\]
\end{example} 

The reason that polarization of a monomial ideal $I =  \langle m_1,\dots,m_r \rangle$ preserves algebraic properties is based on the following fact. The monomial ring $R/I$ is the quotient of $S/\sP(I)$ by the ideal $D$ generated by the regular sequence  $\{ y_{j l} - x_{j}  \  \vert  \ 1 \leq j \leq n, 2\leq l \leq a_j \} $ where $ a_j$ is the maximum power of $x_j$ dividing some minimal generator $m_i$. We recall that a sequence $f_1,\dots,f_t$ of elements in a ring $R$ is a \textbf{regular sequence} if $f_i$ is a nonzerodivisor on $R/ \langle f_1,\dots,f_{i-1} \rangle$ for $1\leq i \leq t$. This passage from  $S/\sP(I)$ to  $R/I$ is called \textbf{depolarization}. Quotienting out by an ideal generated by a regular sequence, especially one generated by homogeneous forms, and hence depolarizing, preserves many algebraic properties. Note that if $x_j^2$ does not divide any minimal generator, there will be no corresponding $y_{jl}$ for $l\geq 2$. Then
one gets
\[ R / I \cong S /(\mathcal{P}(I) + D). \]


Many properties of squarefree polarized ideal $\mathcal{P}(I)$ transfer to the original ideal $I$. For example, 
\begin{itemize} 
\item The minimal free resolution of $R/I$ is obtained from the minimal free $S$-resolution of $S/ (\mathcal{P}(I) + D)$ by depolarization;
\end{itemize}
For given monomial ideals $I$ and $J$ in $R$,
\begin{itemize} 
\item $\mathcal{P}(I + J) = \mathcal{P}(I) + \mathcal{P}(J)$;
\item $\mathcal{P}(I)$ and $I$ are the same height;
\end{itemize}
For the proofs and more properties of polarization of monomial ideals see \cite{Faridi}.
The next theorem shows that we can verify the Cohen-Macaulayness of the monomial ideal using the squarefree monomial ideals via the polarization operation. 
 
\begin{theorem}[Fr\"oberg \cite{Froberg} ] For a given monomial ideal $I = \langle m_1,\dots,m_r \rangle$ in $R = K[x_1,\dots,x_n]$,
$R/I$ is Cohen-Macaulay (Gorenstein) if and only if  $S/ \mathcal{P}(I)$ is Cohen-Macaulay (Gorenstein) where $S = \mathcal{P}(R)$ as defined above.
\end{theorem}

\section{The Polarization Operation}

Let $R=\FF[x_1,\ldots, x_n]$ and $S = \FF[x_1,\ldots, x_n,y_1,\ldots, y_n]$. We want each $y_i$ to act as an alias for $1 - x_i$, and we encode this by defining the \textbf{depolarization ideal},
\[
D = \langle\{x_i + y_i - 1\,|\,i\in [n]\}\rangle,
\]
so that $S/D = R$. We denote the corresponding quotient map by $\pi : S \to S/D = R$, so that $\pi$ identifies $y_i$ and $1 - x_i$ for each $i$.

For a pseudomonomial $f = \prod_{i\in\sigma}x_i\prod_{j\in\tau}(1-x_j) \in R$, we define its \textbf{polarization} to be the squarefree monomial
\[
\sP(f) = \prod_{i\in\sigma}x_i\prod_{j\in\tau}y_j \in S.
\]

\begin{lemma}\label{lemma-divisiblity}
The polarization operation on pseudomonomials, $\sP$ has the property that if $f, g \in R$ are pseudomonomials then $f | g \Leftrightarrow\sP(f) | \sP(g)$ in $S$.
\end{lemma}
\begin{proof}
Let
\[
f = \prod_{i\in\sigma_f}x_i\prod_{j\in\tau_f}(1-x_j), \ \ g = \prod_{i\in\sigma_g}x_i\prod_{j\in\tau_g}(1-x_j).
\]
The condition that $f|g$, is equivalent to the condition that $\sigma_f \subset \sigma_g$ and $\tau_f \subset \tau_g$. The condition that $\sP(f) | \sP(g)$, that is, that
\[
\prod_{i\in\sigma_f}x_i\prod_{j\in\tau_f}y_j \text{ divides } \prod_{i\in\sigma_g}x_i\prod_{j\in\tau_g}y_j
\]
is the same. Thus, $f | g$ if and only if $\sP(f) | \sP(g)$.

\end{proof}

Let $I \subset R$ be a pseudomonomial ideal. As above, $I$ can be generated by the set of minimal pseudomonomials in $I$, or by the set of all pseudomonomials contained in $I$.

We define the polarization of $I$ by showing that these two characterizations  of $I$ are compatible with polarization in the following way.

\begin{theorem}\label{theorem-conditions-polarization}
For a pseudomonomial ideal $I = \langle f_1,\ldots, f_\ell\rangle \subset R$ in canonical form, and a squarefree monomial ideal $J \subset S$, the following are equivalent:
\begin{enumerate}
\item $J$ is the smallest ideal in $S$ such that, for any pseudomonomial $f$ in $R$, $f \in I$ if and only if $\sP(f) \in J$.
\item $J = \langle\sP(f_1),\ldots, \sP(f_\ell)\rangle \subset S$.
\end{enumerate}
\end{theorem}

\begin{proof}

\textbf{(1)$\Rightarrow$(2):}
Suppose that $J_0 \subset S$ is the smallest ideal in $S$ such that for every pseudomonomial $f \in R$, $f \in I \Leftrightarrow \sP(f) \in J_0$. Then 
\[J_0 = \langle \{\sP(f)\ | \ f \in I, f\text{ is a pseudomonomial}\}\rangle.\]
 Let $J = \langle\sP(f_1),\ldots, \sP(f_\ell)\rangle \subset S$, where the $f_i$ are the set of \textit{all} minimal pseudomonomials in $I$. Then it suffices to show that $J_0 = J$.

Obviously $J \subseteq J_0$ because the generators of $J$ are a subset of the generators of $J_0$.

Recall from Lemma~\ref{lemma-divisiblity} that $\sP$ preserves divisibility. Let $f \in I$ a pseudomonomial. Since $f$ is a pseudomonomial in $I$, some factor of $f$ is a minimal pseudomonomial in $I$. Denote it by $f_i$ so that $f_i | f$. Then $\sP(f_i) | \sP(f)$, so $\sP(f_i) | \sP(f)$, and $\sP(f) \in J$. Since $f$ was an arbitrary pseudomonomial in $I$, this shows that every generator of $J_0$ is contained in $J$ and $J_0 \subseteq J$. Therefore, $J=J_0$.

\textbf{(2)$\Rightarrow$(1):}
Let $J = \langle \sP(f_1),\ldots, \sP(f_\ell)\rangle \subset S$, where the $f_i$ are the set of \textit{all} minimal pseudomonomials in $I$. We need to show that $J$ is the smallest ideal in $S$ such that for every pseudomonomial $f \in R$, $f \in I \Leftrightarrow \sP(f) \in J$.

First we show that indeed, for every pseudomonomial $f \in R$, $f \in I \Leftrightarrow \sP(f) \in J$. From the above, if a pseudomonomial $f \in I$, then $\sP(f) \in J$. Suppose $f$ is a pseudomonomial not in $I$. Suppose, for the sake of contradiction, that $\sP(f) \in J$. Then, since $J$ is a monomial ideal $\sP(f)$ is a monomial, some generator $\sP(f_i)$ of $J$ divides $\sP(f)$. But by Lemma~\ref{lemma-divisiblity}, if $\sP(f_i) | \sP(f)$, then $f_i | f$, so $f \in I$, a contradiction. This concludes the proof.

\end{proof}

Now that we have shown that these two definitions produce the same ideal, we can use them equivalently to refer to a well-defined polarization $\sP(I)$ of any pseudomonomial ideal $I \subset R$.

\begin{definition} For a psuedomonomial ideal $I$, we define {\bf the polarization of $I$} to be the squarefree monomial ideal specified by the equivalent conditions of Theorem~\ref{theorem-conditions-polarization}.
\end{definition}

\begin{proposition}  \label{prop-primes} Let $I$ be a pseudomonomial ideal in canonical form, and $\mathbf{p}$ a pseudomonomial prime. 
\begin{itemize}
\item[(i)] $I \subseteq \mathbf{p}$ if and only if $\sP(I) \subseteq \sP(\mathbf{p})$. Moreover, $\sP(\mathbf{p})$ is a prime generated by a subset of the variables.
\item[(ii)] $I \subseteq \mathbf{p}$ if and only if every member of the canonical form is divisible by a generator of the pseudomonomial prime $\mathbf{p}$.
\end{itemize}
\end{proposition}

\begin{proof} Let $I =\langle f_1,\dots,f_r \rangle$ be in canonical form. By Theorem \ref{theorem-conditions-polarization}, $\sP(I) = \langle \sP(f_1),\dots,\sP(f_r) \rangle$. On the other hand, $I\subset \mathbf{p}$ so $f_i \in \mathbf{p}$. Then similarly $f_i\in \mathbf{p}$ if and only if $\sP(f_i)  \in \sP(\mathbf{p})$. Furthermore, by definition $\mathbf{p}$ is generated by $x_i$ or $1-x_j$ for some $i,j \in [n]$, $i\neq j$ , so polarization simply replaces $1-x_j$ by $y_j$ so $\sP(\mathbf{p})$ is a prime generated by the corresponding variables $x_i$ and $y_j$ for $i,j \in [n]$, $i\neq j$. This concludes part (i).

Part (ii) follows by part (i), Lemma~\ref{lemma-divisiblity} and the fact that a monomial ideal is contained in a monomial prime if and only if every generator of the monomial ideal is divisible by a generator of the monomial prime. 
 \end{proof}
 
To show that the operation of polarization preserves algebraic properties, we will use the following theorem.

\begin{theorem}\label{thm-reg-seq}
Let $J_\sC\subseteq R$ be a neural ideal. The sequence 
\[x_1 + y_1 - 1,\ldots, x_n + y_n - 1\]
 is a regular sequence on $S/\sP(J_\sC)$.
\end{theorem}

\begin{proof}
The statement is equivalent to the claim that for all $1\leq t \leq n$, the image of the $x_t + y_t - 1$ is a nonzerodivisor on $S/J_t$, where
\[
J_t = D_t + \sP(J_\sC), \text{ and } D_t = ( x_{t+1} + y_{t+1} - 1, \ldots, x_n + y_n - 1).
\]
We observe first that under the map $D_t\subseteq S$, the image $J'_t$ of $J_t$ is a pseudomonomial ideal, namely the ideal obtained from $\sP(J_\sC)$ by replacing the $y_i$ variables by $1-x_i$ for $i>t$. By Theorem~\ref{thm-PM-PD}, each minimal prime of the ideal $J'_t$ in $S' = R[y_1,\dots,y_t]$ is generated by irreducible pseudomonomial elements; that is $x$'s, $y$'s, $(1-x)$'s, and $(1-y)$'s.

We claim that each minimal pseudomonomial prime of $J'_t$ is generated by elements of the form $x_i$ with $1\leq i \leq n$, $1-x_j$ with $j>t$, or $y_k$ with $k\leq t$. 
By Proposition~\ref{prop-primes}, if a pseudomonomial prime over a pseudomonomial ideal has a generator that does not divide any element, that generator can be removed and the resulting prime still contains the specified ideal. Now, since $J'_t$ is generated by pseudomonomials that are multiples of $x_i$, $1-x_j$ with $j>t$, and $y_k$ with $k\leq t$, the claim follows.

As a consequence of the claim, $x_t+y_t-1$ is a nonzerodivisor on $S'/J'_t$, since it cannot be contained in a prime generated by elements of the form $x_i$, $1-x_j$ with $j>t$, and $y_k$ with $k\leq t$. Then the theorem follows by the isomorphism $S/J_t\cong S'/J'_t$.
\end{proof}

\section{Free Resolutions and Cohen-Macaulayness}

In this section, we apply Theorem~\ref{thm-reg-seq} to determine some algebraic properties of neural ideals. We first consider free resolutions.

Let $A$ be a ring, $I$ an ideal, and $f_1,\dots,f_t$ is a sequence of elements of $A$ whose images form a regular sequence on $A/I$. Let $P_{\bullet}$ be a free resolution of $A/I$ as an $A$-module. The homology of $P_{\bullet} \otimes_A A/ \langle f_1,\dots,f_t \rangle$ can be computed by $\mathrm{Tor}^A_{\bullet}(A/I , A/ \langle f_1,\dots,f_t \rangle )$, which vanishes by the assumption that the images of the $f$'s form a regular sequence on $A/I$. Thus, $P_{\bullet} \otimes_A A/\langle f_1,\dots,f_t\rangle$ gives a free resolution of $A/(I+\langle f_1,\dots,f_t \rangle)$ as an $A/I$-module. The following corollary then follows immediately from Theorem~\ref{thm-reg-seq}.

\begin{corollary}\label{cor-free-res} Let $\sC$ be a code, $J_\sC\subseteq R$ its neural ideal, and $\sP(J_\sC)\subseteq S$ its polarization. Given a free resolution $P_\bullet$ of the squarefree monomial ideal $\sP(J_\sC)$, the complex $P_\bullet \otimes_S S/D$ is a free resolution of $J_\sC$.
\end{corollary}

 Concretely, the complex $P_\bullet \otimes_S S/D$ is the complex of free $R$-modules obtained from $P_\bullet$ by replacing $y_i$ by $1-x_i$ in each of the matrices.

We single out two special types of free resolutions $P_\bullet$ of the squarefree monomial ideals $\sP(J_\sC)$. First, one has the Taylor complex, which is easy to construct, but rarely minimal. Corollary~\ref{cor-free-res} indicates that an analogue of the Taylor complex provides a free resolution for $J_{\sC}$.

\begin{definition} We define the \textbf{Taylor resolution} of the neural ideal $J_\sC$ as follows. Let $CF(J_{\sC})=\langle a_1,\dots,a_t \rangle$. For a subset $H\subseteq \{1,\dots,t\}$, let $M_H=\mathrm{lcm}\{ a_i \ | \ i\in H\}$. For $0\leq i \leq t$, let $F_i=\oplus_{|H|=i} R \cdot e_H$ be the free module generated by the symbols $e_H$ for all $H\subseteq \{1,\dots,t\}$ of cardinality $i$. Let $d_i:F_i\rightarrow F_{i-1}$ be the $R$-linear map such that 
\[d_i(e_H)=\sum_{h\in H} \varepsilon(H,h) \frac{M_H}{M_{H -\{h\}}} e_{H -\{h\}}, \ \text{where} \ \varepsilon(H,h)=(-1)^{\# \{ j\in H \ | \ j<h \}} .\] 
\end{definition}

\begin{proposition} The Taylor resolution of a neural ideal $J_\sC \subset R$ is a free resolution of $R/J_{\sC}$ as an $R$-module.
\end{proposition}

\begin{proof} This is an immediate consequence of Corollary~\ref{cor-free-res}, because the Taylor resolution of $R/J_{\sC}$ is the result of taking the Taylor resolution of $\sP(J_\sC)$ and tensoring with the quotient by the depolarization ideal.
\end{proof}

\begin{example} \label{ex-Taylor-res} 
 Let $\sC = \{000, 100, 111\}$; this is the code A18 in \cite{CIVY}. The neural ideal $J_\sC$ is $\langle x_2(1-x_1), x_2(1-x_3), x_3(1-x_1), x_3(1-x_2)\rangle$ in ${R= \FF_2[x_1,x_2,x_3]}$, thus its polarization ${\sP(J_\sC) = \langle x_2y_1, x_2y_3, x_3y_1, x_3y_2\rangle}$ in ${S= \FF_2[x_1,x_2,x_3,y_1,y_2,y_3]}$. The following is the Taylor resolution of $S/\sP(J_\sC)$;
\[\minCDarrowwidth25pt\begin{CD}
S @<\text{$\tiny \begin{bmatrix} x_2y_1& x_2y_3&  x_3y_1& x_3y_2\end{bmatrix}$}<< S^4 @<d_2<< S^6 @<d_3<<  S^4@< \text{$\tiny\begin{bmatrix} y_3 \\ y_2 \\ -1\\ -1\end{bmatrix}$}<< S @<<< 0
\end{CD}
\]
where $d_2 =\small \begin{bmatrix} -y_3 & -x_3 & -x_3y_2& 0 & 0 &0 \\ y_1 & 0 & 0 & -x_3y_1 & -x_3y_2 & 0\\ 0 &x_2 & 0 & x_2y_3 & 0 & -y_2\\ 0 & 0 & x_2y_1& 0 & x_2y_3 & y_1 
\end{bmatrix}$ and $d_3 = \small \begin{bmatrix} 0 & x_3&x_3y_2 & 0\\ y_2&-y_3&0&0 \\ -1&0&-y_3&0 \\  0&1&0&y_2 \\ 0&0&y_1&-y_1 \\ x_2&0&0&x_2y_3 \end{bmatrix}$.

Then replacing each $y_i$ by $1-x_i$ in the entries of matrices gives the Taylor resolution for the neural ideal $J_\sC$ over $R$.
\end{example}

The other free resolution of $\sP(J_\sC)$ that we concern ourselves with is its minimal resolution. Neural ideals are not homogeneous, so there is no notion of minimal resolution for these ideals. However, Corollary~\ref{cor-free-res}  allows us to define a natural substitute.

\begin{definition} We define the \textbf{canonical resolution} of the neural ideal $J_\sC$ as the complex $P_{\bullet} \otimes_S  S/D$, where $P_{\bullet}$ is the minimal resolution of $\sP(J_\sC)$.\end{definition}

\begin{example} Consider the same neural ideal $J_\sC$ and its polarization $\sP(J_\sC)$ in Example \ref{ex-Taylor-res}. The minimal free resolution of  $S/\sP(J_\sC)$ is 
\[\minCDarrowwidth10pt\begin{CD}
 S @<[x_2y_1\ x_2y_3\  x_3y_1\ x_3y_2]<< S^4 @<\text{$\tiny \begin{bmatrix} -y_3 & -x_3 & 0 & 0 \\ y_1 & 0 & x_3y_2&0\\ 0 & x_2 & 0 & y_2\\ 0 &0 & x_2y_3 & y_1 
\end{bmatrix}$}<< S^4 @< \text{$\begin{bmatrix} x_3y_2 \\ -y_2y_3 \\ y_1\\ -x_2y_3\end{bmatrix}$}<<  S @<<< 0
\end{CD}
\]
Then tensoring the above minimal resolution with the quotient by the depolarization ideal gives the canonical resolution of the neural ideal $J_\sC$ over $R$;
\[\minCDarrowwidth10pt\begin{CD}
 0 @<<< R/J_{\sC} @<<< R @<d_1<< R^4 @<d_2<< R^4 @< d_3<<  R @<<< 0
\end{CD}
\]
where
 
 \[d_1 = \small\begin{bmatrix} x_2(1-x_1)& x_2(1-x_3)&  x_3(1-x_1)& x_3(1-x_2)\end{bmatrix},\]
 \[d_2 =\small \begin{bmatrix} -(1-x_3) & -x_3 & 0 & 0 \\ 1-x_1 & 0 & -x_3(1-x_2)&0\\ 0 & x_2 & 0 & -(1-x_2)\\ 0 &0 & x_2y_3 & (1-x_1) \end{bmatrix},\] 
 and 
\[d_3 =\small\begin{bmatrix} x_3(1-x_2) \\ -(1-x_2)(1-x_3)\\ (1-x_1)\\ -x_2(1-x_3)\end{bmatrix}.\]
\end{example}

We now turn our attention to the Cohen-Macaulay property of $R/J_\sC$. Recall that a local or graded ring $A$ is Cohen-Macaulay if there is a regular sequence $f_1,\dots,f_n$ on $A$ with $n=\mathrm{dim}(A)$. If $A$ is not local or graded, then $A$ is Cohen-Macaulay if its localizations at each of its maximal ideals is Cohen-Macaulay. In particular, every zero-dimensional ring is Cohen-Macaulay.

If $A$ is Cohen-Macaulay and $f_1,\dots,f_n$ is a regular sequence on $A$, then $A/\langle f_1,\dots,f_n\rangle$ is also Cohen-Macaulay. Thus, we have the following.

\begin{corollary}\label{cor-CM}  Let $\sC$ be a code, $J_\sC\subseteq R$ its neural ideal, and $\sP(J_\sC)\subseteq S$ its polarization. If $S/\sP(J_\sC)$ is Cohen-Macaulay, then  so is $R/J_\sC$.
\end{corollary}

The advantage of this corollary is that $\sP(J_\sC)$ is a squarefree monomial ideal, so one can determine whether it is Cohen-Macaulay via the topology of its Stanley-Reisner complex. Unfortunately, the converse of Corollary~\ref{cor-CM} fails.

\begin{example} Let $\sC$ be the code $\{ 000,110,011,101 \}$; this is the code E4 in \cite{CIVY}. For this code, $R/J_\sC$ is zero-dimensional, hence Cohen-Macaulay. The polarization $S/\sP(J_\sC)$ has some components of dimension three, and some of dimension four. Since $S/\sP(J_\sC)$ is graded and not equidimensional, it is not Cohen-Macaulay.
\end{example}

\section{The Polar Complex}

In this section we consider the primary decomposition the polarization of a pseudomonomial ideal. As the polar neural ideal is a squarefree monomial ideal, there is a simplicial complex associated to it via the Stanley-Reisner correspondence. We define the \textbf{polar complex} of a code $\sC$ to be the Stanley-Reisner complex of the polar ideal $\sP(I_{\sC})$. By the Stanley-Reisner correspondence, to describe the minimal primes of the polar ideal is equivalent to describing the facets of the polar complex. To these equivalent ends, we start by introducing some notation and definitions. For a subset $W$ of the set of $2n$ variables $\{x_1,\ldots, x_n, y_1,\ldots,y_n\}$ we define the following index sets 
\begin{itemize}
\item[ ] $x(W) :=\{ i \ \vert \ x_i \in W, \, y_i \notin W\}$, \ \ and  \ \ $y(W) :=\{ i \ \vert \  y_i \in W, \, x_i \notin W\}$, 
\item[ ] $b(W) :=\{ i \ \vert \  x_i \in W, \, y_i \in W\}$, \ \ and  \ \ $n(W) :=\{ i \ \vert \  x_i \notin W, \, y_i \notin W\}$.
\end{itemize}
Notice that these index sets are disjoint by definition.  

For a code $\sC$ and a subset $S\subseteq [n]$, the \textbf{quotient code} $\sC / S$ denotes the set of codewords $\bar{c} \in\{0,1\}^{[n]\setminus S}$ such that there is a codeword $c \in \sC$ with $c_j = \bar{c}_j$ for all $j\notin S$. The quotient code of $\sC$ by $S$ is the same as the code obtained from $\sC$ by deleting the neurons $S$ in the terminology of \cite{Curto-Youngs}.

Let $\mathbf{q}_W$ denote the prime ideal that is generated by the variables in the subset $W$. By the definition of $\mathbf{q}_W$, it is evident that $W' \subseteq W$ if and only if $\mathbf{q}_{W'} \subseteq \mathbf{q}_W$.

To a subset $W$ of the set of $2n$ variables $\{x_1,\dots, x_n, y_1,\dots,y_n\}$ we define \textbf{the interval associated to $W$} to be the boolean interval $V_W \subseteq \{0,1\}^{[n]\setminus b(W) }$ given by
\[ V_W = \{ c \in \{0,1\}^{[n]\setminus b(W) } \ | \ c_i=0 \ \text{for all} \ i\in x(W), \ c_j=1 \ \text{for all} \ j\in y(W) \}. \]

Given this notation, we can characterize the monomial primes containing the polar ideal.

\begin{theorem} 
The monomial prime $\mathbf{q}_W$ contains the polar ideal $\sP(J_{\sC})$ if and only if $V_W$ is an interval of $\sC/ b(W)$.
\end{theorem}

\begin{proof} First we consider the case when the given subset $W$ has $b(W) = \emptyset$, so $\sC/b(W) = \sC$. Therefore, $V_W$ is the same interval as $V_\alpha$ corresponding to $\alpha \in \{ 0,1, *\}^{n}$ with $$\alpha_i = \begin{cases}  0 & \text{ if  } i \in x(W) \\ 
1  & \text{ if  } i \in y(W) \\ 
* & \text{ if  } i \in n(W) 
\end{cases}$$
as in Section~2. By Theorem~\ref{thm-PM-PD} the pseudomonomial neural ideal $J_\sC$ is contained in the pseudomonomial prime $\mathbf{p}_\alpha$ if and only if the interval $V_\alpha=V_W$ is contained in $\sC$. By the assumption $b(W)=\emptyset$, we are in the setting of Proposition~\ref{prop-primes}(i), so $\mathbf{p}_\alpha$ contains $J_\sC$ if and only if $\mathbf{q}_W  = \sP(\mathbf{p}_{\alpha})$ contains $\sP(J_{\sC})$, which verifies this case.

Now let $W$ be an arbitrary subset of  $\{x_1,\dots, x_n, y_1,\dots,y_n\}$. Write $\mathbf{q}_W=\mathbf{m}_{b(W)} + \mathbf{q}_{W'}$, where $\mathbf{m}_{b(W)} =\langle \{x_i, y_i \ | \ i \in b(W) \}\rangle$ and $\mathbf{q}_{W'}$ is generated by the remaining variables in $W$. Note that $W'$ is a subset of the variables $\{x_i, y_i \ | \ i\in [n] \setminus b(W) \}$ and $b(W')=\emptyset$. Now, $\mathbf{q}_W \supseteq \sP(J_{\sC})$ if and only if $(\mathbf{m}_{b(W)} + \mathbf{q}_{W'} )/\mathbf{m}_{b(W)} \supseteq (\mathbf{m}_{b(W)} + \sP(J_{\sC}) )/\mathbf{m}_{b(W)}$. Identifying $\FF_2[x_i, y_i \ | \ i \in [n] \, ] / (\mathbf{m}_{b(W)})$ with $S'=\FF_2[x_i , y_i \ | \ i \in [n] \setminus b(W) ]$, and the ideals with their images in $S'$, the previous containment is equivalent to $\mathbf{q}_{W'} \supseteq \mathcal{J}$, where $\mathcal{J}$ is the monomial ideal generated by the monomial generators of $\sP(J_{\sC})$ that are not divisible by any variable with index in $b(W)$.

We claim that $\mathcal{J}=\sP(J_{\sC / b(W) })$. Indeed, by \cite[1.7.3]{Curto-Youngs}, given a code $\sC$, the canonical form of $\sC/\{i\}$ consists of the elements of the canonical form of $\sC$ that are not divisible by either $x_i$ or $y_i$; the analgous statement for a set of neurons follows immediately. The claim then follows from the description of the polarization via canonical forms. 

We conclude the proof. We have that  $\mathbf{q}_W \supseteq \sP(J_{\sC})$ if and only if $\mathbf{q}_{W'} \supseteq \sP(J_{\sC / b(W) })$ in $S'$. Since $b(W')=\emptyset$, by the case established in the first paragraph, $\mathbf{q}_{W'} \supseteq \sP(J_{\sC / b(W) })$ if and only if $W'$ is an interval of ${\sC / b(W) }$, as required.
\end{proof}

 Then the Stanley-Reisner correspondence says the following
 \begin{corollary}[Stanley-Reisner correspondence-Polar complex]\textit{ $W$ is a face of the complex if and only if  the interval ``anything" in $b(W)$ indices, 1's in $x(W)$ indices, 0's in $y(W)$ indices is an interval of $\sC/n(W)$.}
 \end{corollary}
 \begin{proof}
 This follows from the fact that faces of the Stanley-Reisner complex are complements of sets of variables that form monomial primes containing the Stanley-Reisner ideal.
 \end{proof}
 
 We may also characterize the minimal primes of $\sP(J_{\sC})$.
 
 \begin{corollary} The minimal primes of $\sP(J_{\sC})$ are the primes $\mathbf{q}_W$ such that 
 	\begin{itemize}
 		\item $V_W$ is a maximal interval of the quotient code $\sC / b(W)$, and
 		\item  the boolean interval in $\{0,1\}^{[n]\setminus b(W) \cup \{j\} }$ consisting of elements of $V_W$ in the $[n]\setminus b(W)$ positions and a fixed constant value in the $j$ position does not belong to  $\sC / (b(W) \setminus \{j\})$ for any $j\in b(W)$.
 	\end{itemize}
 \end{corollary}
 \begin{proof} If $W$ corresponds to a nonminimal prime $\mathbf{q}_W$, then it is possible to remove a variable from $W$ to get $W'$ and still have that $\mathbf{q}_{W'}$ contains the polar ideal. If the variable is an $x_i$ or $y_i$ such that the other is not in $W$, then $V_{W'}$ is an interval of $\sC / b(W)$ properly containing $V_W$. If the variable is an $x_i$ or $y_i$ such that the other is in $W$, then $V_{W'}$ consists of elements of $V_W$ plus a fixed value in the $j$ position does not belong to  $\sC / (b(W) \setminus \{j\})$ for any $j\in b(W)$.
 \end{proof}
 
 In particular, maximal intervals of $\sC$ correspond to minimal primes of $\sP(J_{\sC})$, but not every minimal prime arises in this way.
 
 We illustrate this primary decomposition in an example.
 \begin{example} Consider the neural ideal 
 \[J_\sC = \langle x_1x_3, x_3(1-x_2), x_2(1-x_1)(1-x_3)\rangle\] in canonical form, corresponding to the code $\sC = \{000,100,110,011\}$; this is neural code B5 in \cite{CIVY}. Its polar ideal is $ \sP(J _{\sC})  = \langle x_1x_3, y_2x_3, y_1x_2y_3\rangle $. We depict the code as a subset of the Boolean lattice, with codewords in black and noncodewords in red: 
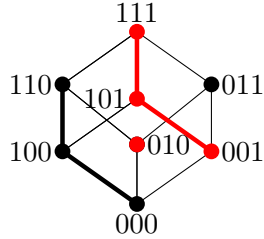
\begin{center}
\begin{tikzpicture}[scale=0.7, tdplot_main_coords]
\coordinate (100) at (0,0,0);
\coordinate (101) at (0,2,0);
\coordinate (111) at (0,2,1.8);
\coordinate (110) at (0,0,1.8);
\coordinate (000) at (2,0,0);
\coordinate (001) at (2,2,0);
\coordinate (011) at (2,2,1.8);
\coordinate (010) at (2,0,1.6);

\node[left] at (100) {100};
\node[left] at (101) {101};
\node[above] at (111) {111};
\node[left] at (110) {110};
\node[below] at (000) {000};
\node[right] at (001) {001};
\node[right] at (011) {011};
\node[right] at (010) {010};

\draw(111)--(110)--(010)--(011); 
\draw  (001)--(101)--(100)--(000);
\draw(110)--(100)--(101)--(111);
\draw(111)--(011)--(001)--(000)--(010);

\draw (111)--(110)--(100)--(000)--(010)--(110);
\draw[black, line width=0.6mm]  (000)--(100)--(110);
\draw[red, line width=0.6mm]  (001)--(101)--(111);

\foreach \position in {(100), (011),  (000),  (010),(110)}
 \fill \position circle (0.15cm); 
 \foreach \position in {(101), (001), (111), (010)}
 \fill \position[red]  circle(0.15cm); 
\end{tikzpicture}
\end{center}

The isolated codeword $011 \in \sC$ corresponds to the subset $W_1 = \{ x_1,y_2,y_3\}$ gives the minimal prime $\mathbf{q}_{W_1} = \langle x_1, y_2, y_3\rangle$. We also have intervals $[*\,0\,0]$ and $[1\,*\,0]$ contained in $\sC$, yielding minimal primes $\mathbf{q}_{W_2} = \langle x_2, x_3\rangle$ and $\mathbf{q}_{W_3} = \langle y_1, x_3\rangle$.  The subsets $W_1,W_2, W_3$ of the variable set $\{ x_1,x_2, x_3,y_1,y_2,y_3\}$ are the only subsets with $b(W_i) = \emptyset$ that correspond to primes containing $\sP(J_\sC)$. 


Now we examine the minimal subsets $W$ whose index set $b(W)  = \{i\}$ for each $i = 1,2,3$. We use the symbol ``$-$'' as a placeholder for an omitted coordinate for intervals in the quotient codes.

 
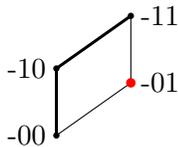
\begin{tikzpicture}[scale=0.7, tdplot_main_coords]
\coordinate (100) at (0,0,0);
\coordinate (101) at (0,2,0);
\coordinate (111) at (0,2,1.8);
\coordinate (110) at (0,0,1.8);

\node[left] at (100) {-00};
\node[right] at (101) {-01};
\node[right] at (111) {-11};
\node[left] at (110) {-10};

\draw(100) -- (101) -- (111) -- (110) -- cycle;
\draw[line width=0.4mm] (100) -- (110);
\draw[line width=0.4mm] (110) -- (111);
\foreach \position in {(111), (110)}
 \fill \position circle (0.06cm); 
 \foreach \position in {(100)}
  \fill \position circle (0.06cm); 
 \foreach \position in {(101)}
 \fill \position[red]  circle(0.09cm); 
 
 \node[text width=10cm, anchor=west, right] at (0,6,-2)
{
Clearly $[-\,1\,*]$ and $[-\,*\,0]$ are the maximal intervals in $\sC/\{1\}$, therefore we get primes  $\mathbf{q}_{W_4} = \langle x_1, y_1, y_2\rangle$, and $\mathbf{q}_{W_5} = \langle x_1, y_1, x_3\rangle$ containing $J_{\sC}$. As neither of these intervals give intervals in $\sC$ with ``$-$'' replaced by either $0$ or $1$, these correspond to minimal primes of $J_{\sC}$. };
\end{tikzpicture}

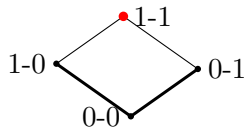
 \begin{tikzpicture}[scale=0.7, tdplot_main_coords]
\coordinate (111) at (0,1.8,1.8);
\coordinate (110) at (0,0,1.8);
\coordinate (011) at (2,1.8,1.8);
\coordinate (010) at (2,0,1.8);

\node[left] at (010) {0-0};
\node[right] at (011) {0-1};
\node[right] at (111) {1-1};
\node[left] at (110) {1-0};

\draw(110) -- (010) -- (011) -- (111) -- cycle;
\draw[line width=0.4mm] (010) -- (011);
\draw[line width=0.4mm] (010) -- (110);
\foreach \position in {(011), (010)}
 \fill \position circle (0.06cm); 
 \foreach \position in {(110)}
  \fill \position circle (0.06cm); 
 \foreach \position in {(111)}
 \fill \position[red]  circle(0.09cm); 
 
  \node[text width=10cm, anchor=west, right] at (0,6,-2)
{
Similarly, $\sC/\{2\}$ contains the intervals $[0\,-\,*]$ and $[*\,-\,0]$ so we get primes $\mathbf{q}_{W_6} = \langle x_2, y_2, x_1\rangle$ and $\mathbf{q}_{W_7} = \langle x_2, y_2, x_3\rangle$ containing $J_{\sC}$. Again, these do not give intervals in $\sC$ with ``$-$'' replaced by $0$ or $1$, so these are minimal primes of $J_{\sC}$.
};
\end{tikzpicture}

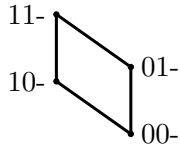
 \begin{tikzpicture}[scale=0.7, tdplot_main_coords]
\coordinate (100) at (0,0,0);
\coordinate (110) at (0,0,1.8);
\coordinate (000) at (2,0,0);
\coordinate (010) at (2,0,1.8);

\node[left] at (100) {10-};
\node[left] at (110) {11-};
\node[right] at (010) {01-};
\node[right] at (000) {00-};

\draw(100) -- (110) -- (010) -- (000) -- cycle;
\draw[line width=0.4mm] (000) -- (100);
\draw[line width=0.4mm] (000) -- (010);
\draw[line width=0.4mm] (110) -- (010);
\draw[line width=0.4mm] (110) -- (100);
\foreach \position in {(100), (010), (000),(110)}
\fill \position circle (0.06cm);

 \node[text width=10cm, anchor=west, right] at (0,5,-2)
{
Finally $\sC/\{3\}$ has the whole interval $[*\,*\,-]$, so we have a prime $\mathbf{q}_{W_8} =\langle x_3, y_3\rangle$; it corresponds to a minimal prime for the same reasoning as the cases above.
};
\end{tikzpicture}

Put together, we get the minimal primary decomposition
\begin{align*}
\sP(J_\sC) &= \langle x_1, y_2, y_3\rangle \cap\langle x_2, x_3\rangle \cap \langle y_1, x_3\rangle \cap \\
 & \langle x_1, y_1, y_2\rangle \cap \langle x_2, y_2, x_1\rangle \cap \langle x_2, y_2, x_3\rangle  \cap \langle x_3, y_3\rangle.
 \end{align*}

Notice that depolarization of the primes $\mathbf{q}_{W_1}, \mathbf{q}_{W_2}$, and $ \mathbf{q}_{W_3}$, simply the ones with no pair $x_i, y_i$ included, are the pseudomonomial primes in the decomposition of the polar ideal $J_\sC$.
\end{example}

\section*{Acknowledgements}

This project was initiated as an REU project at the University of Michigan. The second author was supported in part by the NSF Postdoctoral Research Fellowship
DMS~\#1606353.
The third author was supported by an REU stipend from the NSF grant DMS~\#1306992 of Professor Harm Derksen. 

The authors thank Carina Curto and Anne Shiu for helpful suggestions on a draft of this paper. The second author thanks Alex Kunin for interesting discussions on related forthcoming work.

%
%
%
%
%
%
\bibliographystyle{plain}
	\bibliography{myref}

\end{document}